\documentclass[a4paper,10pt]{article}
\usepackage[utf8]{inputenc}

\usepackage{enumerate}
\usepackage{graphicx} 
\usepackage[T1]{fontenc}
\usepackage{amsmath, amssymb, amsthm}
\usepackage{bbm}
\usepackage[normalem]{ulem}
\usepackage[english]{babel}
\usepackage{eurosym}
\usepackage{fullpage}
\usepackage{palatino}
\usepackage{tikz}
\usepackage{verbatim}
\usepackage{fancybox}
\usepackage{enumitem}
\usepackage{enumerate}
\usepackage{mathtools}
\usepackage{thmtools} 
\usepackage{thm-restate}
\usepackage{url}
\usepackage[unicode]{hyperref}
\usepackage{bookmark}
\usepackage[absolute]{textpos}

\newtheorem{theorem}{Theorem}
\newtheorem{lemma}[theorem]{Lemma}
\newtheorem{corollary}[theorem]{Corollary}
\newtheorem{conjecture}[theorem]{Conjecture}
\newtheorem{claim}[theorem]{Claim}

\newcommand{\Abb}{\mathbf{A}}

\newcommand{\N}{\mathbb{N}}

\newcommand{\Pp}{\mathcal{P}}
\newcommand{\Qq}{\mathcal{Q}}
\newcommand{\Rr}{\mathcal{R}}
\newcommand{\Aa}{\mathcal{A}}
\newcommand{\Bb}{\mathcal{B}}
\newcommand{\Ff}{\mathcal{F}}
\newcommand{\Ss}{\mathcal{S}}
\newcommand{\K}{\mathcal{K}}

\newcommand{\Nn}{\mathcal{N}}
\newcommand{\Oh}{\mathcal{O}}

\newcommand{\size}[1]{{|#1|}}

\newcommand{\tw}{\mathrm{tw}}
\newcommand{\td}{\mathrm{td}}

\newcommand{\wcol}{\mathrm{wcol}}
\newcommand{\scol}{\mathrm{scol}}

\renewcommand{\leq}{\leqslant}
\renewcommand{\geq}{\geqslant}
\renewcommand{\le}{\leqslant}

\renewcommand{\setminus}{-}
\renewcommand{\epsilon}{\varepsilon}

\title{Degeneracy of $P_t$-free and  $C_{\geq t}$-free graphs with no large complete bipartite subgraphs}
\author{
Marthe Bonamy
\thanks{Univ.\ Bordeaux, CNRS,  Bordeaux INP, LaBRI, UMR 5800, F-33400, Talence, France. E-mail: \texttt{marthe.bonamy@u-bordeaux.fr}.\newline Supported by ANR Project GrR (\textsc{ANR-18-CE40-0032}).}
\and Nicolas Bousquet
\thanks{LIRIS, CNRS, Universit\'e Claude Bernard Lyon 1, Université de Lyon, France. E-mail: \texttt{nicolas.bousquet@univ-lyon1.fr}.\newline Supported by ANR Project GrR (\textsc{ANR-18-CE40-0032}).}
\and Micha\l{} Pilipczuk
\thanks{Institute of Informatics, University of Warsaw, Poland. E-mail: \texttt{michal.pilipczuk@mimuw.edu.pl}.\newline
This work is 
a part of project TOTAL that has received funding from the European Research Council (ERC) 
under the European Union's Horizon 2020 research and innovation programme (grant agreement No.~677651).}
\and Pawe\l{} Rz\k{a}\.zewski
\thanks{Faculty of Mathematics and Information Science, Warsaw University of Technology and Institute of
Informatics, University of Warsaw. E-mail: \texttt{p.rzazewski@mini.pw.edu.pl}.\newline
This work is 
a part of project CUTACOMBS that has received funding from the European Research Council (ERC) 
under the European Union's Horizon 2020 research and innovation programme (grant agreement No.~714704).
}
\and St\'{e}phan Thomass\'{e}
\thanks{Univ Lyon, CNRS, ENS de Lyon, Université Claude Bernard Lyon 1, LIP UMR5668, France. E-mail: \texttt{stephan.thomasse@ens-lyon.fr}.}
\and Bartosz Walczak
\thanks{Department of Theoretical Computer Science, Faculty of Mathematics and Computer Science, Jagiellonian University, Krak\'ow, Poland. E-mail: \texttt{walczak@tcs.uj.edu.pl}.\newline
Partially supported by the National Science Centre of Poland grant no.\ 2019/34/E/ST6/00443.}
}

\hypersetup{
  pdftitle={Degeneracy of P\_t-free and C\_\{≥t\}-free graphs with no large complete bipartite subgraphs},
  pdfauthor={Marthe Bonamy, Nicolas Bousquet, Michał Pilipczuk, Paweł Rzążewski, Stéphan Thomassé, Bartosz Walczak},
  colorlinks=true,linkcolor=black,citecolor=black,filecolor=black,urlcolor=black
}

\date{}

\begin{document}

\maketitle

\begin{abstract}
A hereditary class of graphs $\mathcal{G}$ is \emph{$\chi$-bounded} if there exists a function $f$ such that every graph $G \in \mathcal{G}$ satisfies $\chi(G) \leq f(\omega(G))$, where $\chi(G)$ and $\omega(G)$ are the chromatic number and the clique number of $G$, respectively.
As one of the first results about $\chi$-bounded classes, Gy\'{a}rf\'{a}s proved in 1985 that if $G$ is $P_t$-free, i.e., does not contain a $t$-vertex path as an induced subgraph, then $\chi(G) \leq (t-1)^{\omega(G)-1}$.
In 2017, Chudnovsky, Scott, and Seymour proved that  $C_{\geq t}$-free graphs, i.e., graphs that exclude induced cycles with at least $t$ vertices, are $\chi$-bounded as well, and the obtained bound is again superpolynomial in the clique number. 
Note that $P_{t-1}$-free graphs are in particular $C_{\geq t}$-free.
It remains a major open problem in the area whether for  $C_{\geq t}$-free, or at least $P_t$-free graphs $G$, the value of $\chi(G)$ can be bounded from above by a polynomial function of $\omega(G)$.
We consider a relaxation of this problem, where we compare the chromatic number with the size of a largest balanced biclique contained in the graph as a (not necessarily induced) subgraph.
We show that for every $t$ there exists a constant $c$ such that for every $\ell$ and every $C_{\geq t}$-free graph which does not contain $K_{\ell,\ell}$ as a subgraph, it holds that $\chi(G) \leq \ell^{c}$.
\end{abstract}

\begin{textblock}{20}(0, 11.6)
\includegraphics[width=40px]{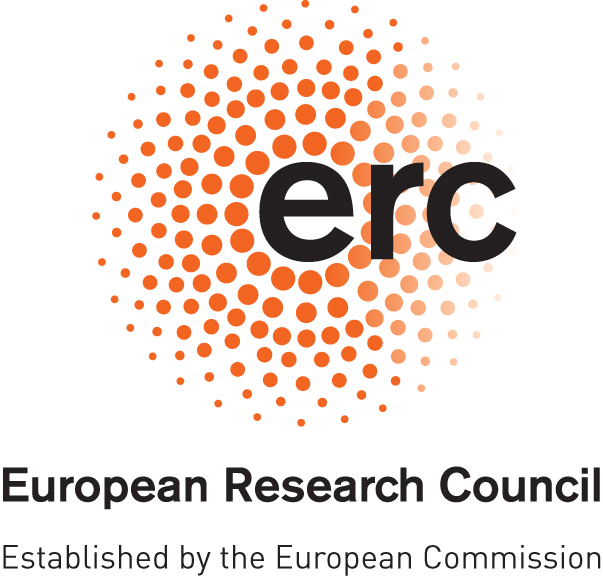}%
\end{textblock}
\begin{textblock}{20}(-0.25, 12)
\includegraphics[width=60px]{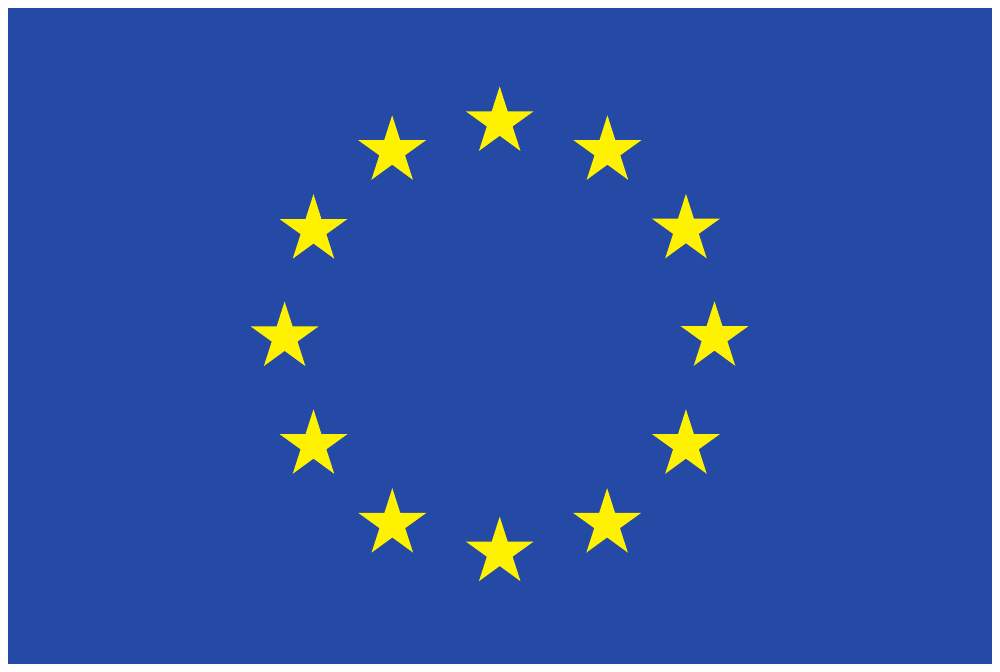}%
\end{textblock}

\section{Introduction}
For an integer $k$, a {\em{$k$-coloring}} of a graph $G$ is a partition of $V(G)$ into $k$ independent sets.
The \emph{chromatic number} of a graph $G$, denoted by $\chi(G)$, is the minimum $k$ for which $G$ admits a $k$-coloring.
The chromatic number is arguably one of best studied graph parameters, and trying to estimate its value for graphs belonging to certain classes gave rise to many important results in graph theory~\cite{appel1976every,chudnovsky2006strong,soifer2008mathematical}.
An obvious lower bound on $\chi(G)$ is $\omega(G)$, the size of the largest \emph{clique} in $G$, i.e., set of pairwise adjacent vertices.
In general, $\chi(G)$ cannot be bounded in terms of $\omega(G)$, as witnessed by various constructions of triangle-free graphs with arbitrarily large chromatic number~\cite{mycielski1955coloriage,zykov1949some}.

However, obtaining such an upper bound might be possible when some additional restrictions are imposed on $G$.
A class of graphs $\mathcal{G}$ is \emph{$\chi$-bounded} if there exists a function $f$, depending on $\mathcal{G}$ only, such that for every $G \in \mathcal{G}$ and every induced subgraph $G'$ of $G$ it holds that $\chi(G') \leq f(\omega(G'))$. The function $f$ is sometimes called a \emph{$\chi$-binding function}. We refer the reader to the recent survey of Scott and Seymour~\cite{DBLP:journals/jgt/ScottS20} for more background.

A rich family of natural graph classes that could be considered in this context can be defined by forbidding certain substructures.
For a fixed graph $H$, a graph $G$ is \emph{$H$-free} if it does not contain an induced subgraph isomorphic to $H$.
One of the central problems regarding $\chi$-boundedness is the following conjecture, formulated independently by Gy\'{a}rf\'{a}s~\cite{gyarfas1975ramsey} and Sumner~\cite{sumner1981subtrees}.

\begin{conjecture}[Gy\'{a}rf\'{a}s--Sumner]\label{GSconjecture}
For every tree $H$, the class of $H$-free graphs is $\chi$-bounded.
\end{conjecture}

Note that if $H$ has a cycle, then $H$-free graphs are not $\chi$-bounded: Erd\H{o}s~\cite{erdos1959graph} proved there exist graphs with arbitrarily large girth and chromatic number. Note that such graphs have clique number 2.

The Gy\'{a}rf\'{a}s--Sumner conjecture is resolved only for very specific trees $H$~\cite{DBLP:journals/jct/ChudnovskyS14,kierstead1994radius,scott1997induced}.
In particular,  Gy\'{a}rf\'{a}s~\cite{gyarfas1985problems} proved that for every $t$, every $P_t$-free graph $G$ satisfies $\chi(G) \leq (t-1)^{\omega(G)-1}$, where $P_t$ denotes the path on $t$ vertices. This upper bound was subsequently improved by Gravier, Ho\`ang, and Maffray~\cite{gravier2003coloring} to $(t-2)^{\omega(G)-1}$.
On the other hand, the best known lower bound on the $\chi$-binding function for $P_t$-free graphs is $f(\omega) =\Omega((\omega/\log \omega)^{(t+1)/4})$, see the discussion in~\cite[Problem 12.2]{DBLP:journals/jgt/ScottS20}.
The question whether $P_t$-free graphs are \emph{polynomially $\chi$-bounded}, i.e., whether the $\chi$-binding function can be chosen to be a polynomial, remains a major open question in the area. It is not even known whether $P_5$-free graphs are polynomially $\chi$-bounded.
On the positive side, very recently Scott, Seymour, and Spirkl proved that $P_5$-free graphs are \emph{quasi-polynomially} $\chi$-bounded~\cite{ScottSS21}.

Interestingly, a weaker variant of the Gy\'{a}rf\'{a}s--Sumner conjecture appears to be true:
For every tree $H$, there is a function $f$ such that every $H$-free graph $G$ that does not contain an induced complete bipartite graph $K_{\ell,\ell}$ satisfies $\chi(G) \leq f(\omega(G),\ell)$~\cite{kierstead1994radius}.
However, the bound is superpolynomial both in terms of $\omega(G)$ and $\ell$.

In this paper, we pursue a similar direction and compare the chromatic number with the maximum size of a \emph{biclique} contained in $G$ (instead of the largest \emph{clique} as in $\chi$-boundedness).
More formally, we consider (superclasses of) $P_t$-free graphs $G$ that do not contain $K_{\ell,\ell}$ as a (not necessarily induced) subgraph and try to bound $\chi(G)$ by a function of $\ell$. 
Let us point out that excluding a large biclique as a subgraph is equivalent to excluding a large clique or a large induced biclique.
Indeed, if a graph has a clique of size at least $2\ell$ or an induced $K_{\ell,\ell}$, then it clearly contains $K_{\ell,\ell}$ as a subgraph.
On the other hand, if a graph contains a $K_{\ell,\ell}$ as a subgraph, then, by Ramsey's theorem, it contains either $K_s$ or an induced $K_{s,s}$,
where $s = \Omega(\log \ell)$. Note that $\ell$ is superpolynomial in $s$.

Actually, we consider a slightly more general setting, where instead of the chromatic number of a graph, we study its degeneracy.
A graph $G$ is \emph{$d$-degenerate} if every induced subgraph of $G$ contains a vertex of degree at most $d$. The \emph{degeneracy} of $G$ is the minimum $d$ for which $G$ is $d$-degenerate.
It is well known that every $d$-degenerate graph is $(d+1)$-colorable, so any upper bound on the degeneracy yields the same upper bound on the chromatic number, up to the additive term of~$1$.

Known results already imply that the degeneracy of $P_t$-free graphs which do not contain $K_{\ell,\ell}$ as a subgraph is bounded by some function of $t$ and $\ell$~\cite{DBLP:conf/swat/AtminasLR12,DBLP:journals/combinatorica/KuhnO04a}; 
we discuss these related results later in the introduction.
However, up to the best of our knowledge, none of the previously known bounds is polynomial in $\ell$.

As our first result, we give a short and elementary proof of the following.

\begin{theorem}\label{thm:p5}
Every $P_5$-free graph that does not contain $K_{\ell,\ell}$ as a subgraph has degeneracy $\Oh(\ell^3)$.
\end{theorem}

Actually, we show a more general result.
For $d \geq 2$, let $S'_d$ denote the $1$-subdivision of $K_{1,d}$, i.e., the graph obtained from a star with $d$ leaves by subdividing each edge once. 
Note that $P_5 = S'_2$, so Theorem~\ref{thm:p5} is a special case of the following result for $d=2$.

\begin{restatable}{theorem}{subdividedstar}
\label{thm:S'_d}
For any $d\geq 2$ and $\ell\geq 2$, every $S'_d$-free graph that does not contain $K_{\ell,\ell}$ as a subgraph has degeneracy $\Oh(\ell^{2d-1})$.
\end{restatable}

Next we consider $P_t$-free graphs for any fixed $t$.
In fact, we again work in a more general setting: we study $C_{\geq t}$-free graphs, i.e., graphs that do not contain any induced cycle with at least $t$ vertices.
Clearly $P_{t-1}$-free graphs are in particular $C_{\geq t}$-free, but on the other hand $C_{\geq t}$-free graphs may have arbitrarily long induced paths.
Already in 1985, Gy\'{a}rf\'{a}s~\cite{gyarfas1985problems} conjectured that $C_{\geq t}$-free graphs are $\chi$-bounded, and this conjecture was confirmed after more than 30 years by Chudnovsky, Scott, and Seymour~\cite{ChudnovskySS17}.
Let us point out that the bound on the chromatic number obtained by the authors is superpolynomial in the clique number.

As the main result of the paper, we show the following bound.

\begin{restatable}{theorem}{mainthm}
\label{thm:main}
There is a function $f\colon \N\to \N$ such that every $C_{\geq t}$-free graph that does not contain $K_{\ell,\ell}$ as a subgraph has degeneracy at most $\ell^{f(t)}$.
\end{restatable}

Since the chromatic number of a graph is bounded by its degeneracy plus one, we conclude the following.

\begin{corollary}\label{coro:main}
There is a function $f\colon \N\to \N$ such that every   $C_{\geq t}$-free graph that does not contain $K_{\ell,\ell}$ as a subgraph has chromatic number at most $\ell^{f(t)}+1$.
\end{corollary}

Very recently, Gartland \emph{et al.}~\cite{deg2tw} showed that every  $C_{\geq t}$-free graph with degeneracy at most $d$ has treewidth bounded by $(dt)^{\Oh(t)}$ (note that in general, treewidth cannot be bounded by a function of degeneracy; e.g., $3$-regular expanders have degeneracy $3$ and treewidth linear in the number of vertices~\cite{grohemarx}).
By combining this result with Theorem~\ref{thm:main}, we immediately obtain the following structural corollary.

\begin{corollary}\label{coro:tw}
There is a function $h\colon \N\to \N$ such that every   $C_{\geq t}$-free graph that does not contain $K_{\ell,\ell}$ as a subgraph has treewidth at most $\ell^{h(t)}$.
\end{corollary}

Actually, in $P_t$-free graphs we can strengthen Corollary~\ref{coro:tw} by considering the parameter \emph{treedepth}, which is bounded from below by the treewidth, but can be much larger in general.
The treedepth of a graph $G$ is the minimum height of a rooted forest on the vertex set $V(G)$ with the property that every pair of vertices adjacent in $G$ is in the ancestor-descendant relation.
For the motivation and several equivalent definitions of this parameter, we refer the reader to the monograph by Ne\v{s}et\v{r}il and Ossona de Mendez~\cite[Section~6.4]{sparsity}.

\begin{theorem}\label{thm:td}
There is a function $h\colon \N\to \N$ such that every $P_t$-free graph that does not contain $K_{\ell,\ell}$ as a subgraph has treedepth at most $\ell^{h(t)}$.
\end{theorem}

Let us point out that such a strengthening is not possible for $C_{\geq t}$-free graphs, as the treedepth of an $n$-vertex path is $\lceil \log(n+1) \rceil$.

Before we proceed to the proofs, we discuss some connections of our results with other known results in graph theory.

\paragraph{Connection 1: Structural results about $P_t$-free and  $C_{\geq t}$-free graphs excluding a fixed biclique.}
As we already pointed out, the influence of the size of a largest complete bipartite subgraph on the structure of graphs with some forbidden induced subgraphs was already studied.
Atminas, Lozin, and Razgon~\cite{DBLP:conf/swat/AtminasLR12} proved that every graph that contains a long (non-necessarily induced) path contains a long induced path or a large biclique as a subgraph. As a corollary to this result, the authors showed that every $P_t$-free graph which does not contain $K_{\ell,\ell}$ as a subgraph has treewidth bounded by a function of $t$ and $\ell$.
The result is proved by a series of applications of Ramsey's theorem and the obtained bound is superpolynomial in $\ell$.
An analogous result for $C_{\geq t}$-free graphs was shown by Wei{\ss}auer~\cite{DBLP:journals/jct/Weissauer19}.
Since the degeneracy is bounded from above by the treewidth plus one, an upper bound on the treewidth implies an upper bound on the degeneracy.

Another result of similar flavor was obtained by K\"uhn and Osthus~\cite{DBLP:journals/combinatorica/KuhnO04a}.
For a fixed graph $H$, we say that $G$ is \emph{$H$-subdivision-free} if it does not contain any subdivision of $H$ as an induced subgraph.
K\"uhn and Osthus~\cite{DBLP:journals/combinatorica/KuhnO04a} proved that for every $H$, graphs that are $H$-subdivision-free and do not contain $K_{\ell,\ell}$ as a subgraph have degeneracy bounded by some function $d(H,\ell)$ of $H$ and $\ell$.
Since the $P_t$-subdivision-free graphs are precisely the $P_t$-free graphs and $C_t$-subdivision-free graphs are precisely  $C_{\geq t}$-free graphs,
the statement above is very close to the statement of our Theorem~\ref{thm:main}.
However, the authors were more concerned about the dependence of $d(H,\ell)$ on the number of vertices of $H$, and their bound is superpolynomial in terms of $\ell$.

\paragraph{Connection 2: Erd\H{o}s--Hajnal conjecture.}
As a corollary to Ramsey's theorem~\cite{ramsey2009problem} (see also Erd\H{o}s and Szekeres~\cite{erdos1935combinatorial}), we obtain that every graph with $n$ vertices contains a clique or an independent set of size $\Omega(\log n)$. Actually, an even stronger statement is true: for almost every graph, the maximum sizes of a clique and of an independent set are of order $\Theta(\log n)$.
However, for some natural graph classes, the bound can be significantly improved.
Erd\H{o}s and Hajnal~\cite{DBLP:journals/dam/ErdosH89} proved that for every fixed graph $H$, there is $\epsilon_H>0$ such that every $H$-free graph with $n$ vertices contains a clique or an independent set of size at least $2^{\epsilon_H \sqrt{\log n}}$.
They also conjectured that a \emph{polynomial} lower bound should hold.

\begin{conjecture}[Erd\H{o}s--Hajnal]
For every graph $H$, there exists $\epsilon_H >0$ such that every $H$-free graph contains an independent set of size at least $n^{\epsilon_H}$ or a clique of size at least $n^{\epsilon_H}$.
\end{conjecture}

We remark that for every fixed $H$, almost no graphs are $H$-free.
The Erd\H{o}s--Hajnal conjecture can be seen as a weakening of polynomial $\chi$-boundedness since if we can prove that $H$-free graphs are polynomially $\chi$-bounded, then in particular the graph $H$ satisfies the Erd\H{o}s--Hajnal conjecture.

Despite considerable interest received by the conjecture in the last decades, we still know very little.
The conjecture is known to hold if $H$ is a cograph~\cite{DBLP:journals/dam/ErdosH89}, has at most four vertices~\cite{DBLP:journals/combinatorica/AlonPS01},  is the five-vertex graph called the \emph{bull}~\cite{DBLP:journals/jct/ChudnovskyS08d}, or a $C_5$~\cite{ChudnovskySSS21}. The conjecture is in particular open for $P_5$-free graphs.
For more information, we refer the reader to the survey by Chudnovsky~\cite{DBLP:journals/jgt/Chudnovsky14}.

Interestingly, Fox and Sudakov~\cite{DBLP:journals/combinatorica/FoxS09} proved a version of the Erd\H{o}s--Hajnal conjecture, 
which is weakened in a similar spirit as our Theorem~\ref{thm:main} weakens the notion of polynomial $\chi$-boundedness.

\begin{theorem}[Fox, Sudakov~\cite{DBLP:journals/combinatorica/FoxS09}]\label{thm:FoxS}
For every graph $H$, there exists $\epsilon_H >0$ such that every $H$-free graph contains an independent set of size $n^{\epsilon_H}$ or a complete bipartite subgraph where each side has size at least $n^{\epsilon_H}$.
\end{theorem}

\paragraph{Connection 3: similarities of $P_t$-free and $C_{\geq t}$-free graphs and string graphs.}
The last connection we would like to mention is an interesting similarity between the classes of $P_t$-free and $C_{\geq t}$-free graphs and the class of \emph{string graphs}~\cite{DBLP:journals/jct/Kratochvil91}.
A graph $G$ is a string graph if it admits a representation that maps each vertex of $G$ to a continuous curve in the plane, and two vertices are adjacent if and only if their corresponding curves intersect. String graphs encapsulate many natural classes of geometric intersection graphs and are widely studied both from the combinatorial and the algorithmic point of view.
String graphs are $H$-subdivision-free whenever $H$ is an $({\geq 1})$-subdivision of a non-planar graph, i.e., it is a graph obtained from a non-planar graph by subdividing each edge at least once.

Even though string graphs and $C_{\geq t}$-free graphs are incomparable with respect to inclusion, it turns out that, somewhat surprisingly, similar algorithmic techniques work in both settings.
It is known that every $C_{\geq t}$-free graph with maximum degree $\Delta$ admits a balanced separator of size $\Oh(\Delta)$, see Bacs\'{o} \emph{et al.}~\cite{BacsoLMPTL19}  and Chudnovsky \emph{et al.}~\cite{ChudnovskyPPT20-1,ChudnovskyPPT20}.
On the other hand, every string graph with $m$ edges has a balanced separator of size $\Oh(\sqrt{m})$, see Matou\v{s}ek~\cite{DBLP:journals/cpc/Matousek14} and Lee~\cite{DBLP:conf/innovations/Lee17}.
Qualitatively, both these results say that a graph either has a vertex of large degree, or has a small balanced separator.
This property turned out to be very useful in the design of algorithms,
as for many natural problems, including \textsc{Max Independent Set}, \textsc{Max Induced Matching}, or \textsc{3-Coloring}, each of the two possible outcomes
allows us to compute the solution efficiently. Such a win-win approach leads to a subexponential running time for the considered problems~\cite{NovotnaOPRLW19,DBLP:journals/jcss/OkrasaR20}.

However, in string graphs a much stronger statement is true: if a string graph does not admit a small balanced separator, then not only it has a vertex of large degree, but it even contains a large biclique. This follows from a combination of the above-mentioned theorem of Lee~\cite{DBLP:conf/innovations/Lee17} and the following structural result.

\begin{theorem}[Fox, Pach~\cite{DBLP:journals/cpc/FoxP14}; Lee~\cite{DBLP:conf/innovations/Lee17}]\label{thm:FoxPach}
Every string graph that does not contain $K_{\ell,\ell}$ as a subgraph has degeneracy $\Oh(\ell \log \ell)$.
\end{theorem}

This property of string graphs was also used algorithmically~\cite{NovotnaOPRLW19,DBLP:journals/jcss/OkrasaR20}.
Our Theorem~\ref{thm:main} is the analogue of Theorem~\ref{thm:FoxPach} in $C_{\geq t}$-free graphs.

\paragraph{Organization of the paper.}
In Section~\ref{sec:p5free}, we prove Theorem~\ref{thm:S'_d}.
Then, in Section~\ref{sec:tools}, we present some auxiliary results and tools.
We present the proof of our main result, i.e., Theorem~\ref{thm:main}, in Section~\ref{sec:main}. Theorem~\ref{thm:td} is proved in Section~\ref{sec:td}.

\section{Proof of Theorem~\ref{thm:S'_d}}\label{sec:p5free}
For $k\in \N$, we let $[k]\coloneqq\{1,\ldots,k\}$.
In this section we prove Theorem~\ref{thm:S'_d}.

\subdividedstar*

Before we proceed to the proof, we show three auxiliary results that will be used later.

\begin{lemma}\label{lem:non-neighbors}
If $G$ is a graph that does not contain $K_{\ell,\ell}$ as a subgraph and $U$ is a set of vertices in $G$ of size at least $\ell \cdot p$, then every vertex in $V(G)\setminus U$, except at most $\ell-1$, has at least $p$ non-neighbors in $U$.
\end{lemma}

\begin{proof}
For contradiction, suppose there are $\ell$ vertices $v_1,\ldots,v_\ell\in V(G)\setminus U$ each with at most $p-1$ non-neighbors in $U$.
Then the number of common neighbors of $v_1,\ldots,v_\ell$ in $U$ is at least $\ell  p-\ell(p-1)=\ell$, contradicting the assumption that $G$ does not contain $K_{\ell,\ell}$ as a subgraph.
\end{proof}

Repeated application of Lemma~\ref{lem:non-neighbors} yields the following.

\begin{lemma}\label{lem:non-neighbors2}
Let $G$ be a graph that does not contain $K_{\ell,\ell}$ as a subgraph.
Let $U_1,\ldots,U_r$ be pairwise disjoint sets of vertices, each of size at least $\ell \cdot p$.
If $|V(G)\setminus \bigcup_{i=1}^r U_i| > r \cdot (\ell-1)$, then there is a vertex in $V(G)\setminus \bigcup_{i=1}^r U_i$ with at least $p$ non-neighbors in each of the sets $U_1,\ldots,U_r$.
\end{lemma}

\begin{proof}
We apply Lemma~\ref{lem:non-neighbors} iteratively for $U=U_i$, for each $i \in [r]$.
In iteration $i$, we discard at most $\ell-1$ vertices of $V(G)\setminus \bigcup_{i=1}^r U_i$ which have fewer than $p$ non-neighbors in $U_i$.
Since $|V(G)\setminus \bigcup_{i=1}^r U_i| > r \cdot (\ell-1)$, we conclude that after the last iteration we still have at least one vertex, which satisfies the conclusion of the lemma.
\end{proof}

\begin{lemma}\label{lem:independent-set}
Let $G$ be a graph that does not contain $K_{\ell,\ell}$ as a subgraph, and let $V_1,\ldots,V_d$ be pairwise disjoint sets of vertices in $G$,
each of size at least $\ell^{d-1}$. Then there exists an independent set $\{v_1,\ldots,v_d\}$, such that $v_i\in V_i$ for all $i \in [d]$.
\end{lemma}

\begin{proof}
For each $i \in [d]$, we construct an $(i-1)$-tuple of pairwise non-adjacent vertices $v_1\in V_1,\:\ldots,\:v_{i-1}\in V_{i-1}$ and sets $V^i_i\subseteq V_i,\:\ldots,\:V^i_d\subseteq V_d$, each of size at least $\ell^{d-i}$,
such that every vertex in $V^i_i\cup\cdots\cup V^i_d$ is non-adjacent to $v_1,\ldots,v_{i-1}$.
We start with $V^1_1=V_1,\:\ldots,\:V^1_d=V_d$, which satisfy the conditions for $i=1$.
Then, for $i \in [d-1]$, we get from $i$ to $i+1$ as follows.
We apply Lemma~\ref{lem:non-neighbors2} for $G[\bigcup_{j=i}^r V^i_j]$, $r=d-i$, $p=\ell^{d-i-1}$, and sets $V^i_{i+1},V^i_{i+2},\ldots,V^i_d$.
Since $\size{V^i_i}\geq \ell^{d-i} > (d-i)(\ell-1)$,
Lemma~\ref{lem:non-neighbors2} yields a vertex $v_i\in V^i_i$ with at least $p=\ell^{d-i-1}$ non-neighbors in each of $V^i_{i+1},\ldots,V^i_d$.
The respective sets $V^{i+1}_{i+1},\ldots,V^{i+1}_d$ of non-neighbors of $v_i$ in $V^i_{i+1},\ldots,V^i_d$ are as required for the next value of $i$.
Since $\size{V^d_d}\geq 1$, we can choose a vertex $v_d\in V^d_d$ to complete the construction.
\end{proof}

Now we are ready to prove Theorem~\ref{thm:S'_d}.

\begin{proof}[Proof of Theorem~\ref{thm:S'_d}]
We show, for each $k \in [\ell]$, that every $S'_d$-free graph not containing $K_{k,\ell}$ as a subgraph has a vertex of degree at most $(k-1)\left (\ell^{d-1}+(d-1)\ell^d+\ell^{2d-2} \right)+\ell-k$, which is $\Oh(\ell^{2d-1})$ when $k=\ell$.
The statement clearly holds for $k=1$.
For the induction step, assume it holds for $k-1$, where $2\leq k\leq\ell$, and let $G$ be an $S'_d$-free graph not containing $K_{k,\ell}$ as a subgraph.
Fix a vertex $r$ of $G$.
Let $A$ be the set of neighbors of $r$ in $G$, and let $B=V(G)\setminus(A\cup\{r\})$.
For a vertex $u\in A$, let $B(u)$ be the set of neighbors of $u$ in $B$, and for a tuple $u_1,\ldots,u_m\in A$, let $B(u_1,\ldots,u_m)=\bigcup_{i=1}^mB(u_i)$.
Let
\begin{equation}\label{eq:U-def}
U=\{u\in A\colon\size{B(u)}\geq\ell^{2d-2}\}\text{.}
\end{equation}

We aim to show that $\size{U}< \ell^{d-1}+(d-1)\ell^d.$ For the sake of contradiction, suppose this is not the case, i.e., 
\begin{equation}\label{eq:U-size}
\size{U}\geq\ell^{d-1}+(d-1)\ell^d\text{.}
\end{equation}
We claim that there is an independent set $\{u_1,\ldots,u_d\} \subseteq U$ of size $d$, such that the following holds for each $i \in [d]$:
\begin{equation}\label{eq:neighbors}
\size{B(u_i)\setminus B(u_1,\ldots,u_{i-1},u_{i+1},\ldots,u_d)}\geq\ell^{d-1}\text{.}
\end{equation}
To this end, for each  $j \in [d]$, we construct a $(j-1)$-element independent set $\{u_1,\ldots,u_{j-1}\} \subseteq  U$ and a set $U_j\subseteq U\setminus\{u_1,\ldots,u_{j-1}\}$ of vertices non-adjacent to any of $u_1,\ldots,u_{j-1}$, such that
\begin{equation}\label{eq:step-size}
\size{U_j}\geq\ell^{d-j}+(d-1)\ell^{d-j+1}-(j-1)\ell
\end{equation}
and the following holds for each $u_j\in U_j$ and each $i \in [j]$:
\begin{equation}\label{eq:step-neighbors}
\size{B(u_i)\setminus B(u_1,\ldots,u_{i-1},u_{i+1},\ldots,u_j)}\geq\ell^{2d-j-1}\text{.}
\end{equation}
We start with $U_1=U$, which satisfies the conditions for $j=1$ by \eqref{eq:U-def} and \eqref{eq:U-size}.
Then, for $j \in [d-1]$, we get from $j$ to $j+1$ as follows.
Let $u_j^1,\ldots,u_j^\ell$ be $\ell$ vertices with least values of $\size{B(u)\setminus B(u_1,\ldots,u_{j-1})}$ among all $u\in U_j$.
Since it follows from \eqref{eq:step-size} that
\[\size{U_j\setminus\{u_j^1,\ldots,u_j^\ell\}}\geq \left( \ell^{d-j}+(d-1)\ell^{d-j+1}-(j-1)\ell \right) - \ell=\ell^{d-j}+(d-1)\ell^{d-j+1}-j\ell\text{,}\]
Lemma~\ref{lem:non-neighbors} implies that at least one of the vertices $u_j^1,\ldots,u_j^\ell$ has at least $\ell^{d-j-1}+(d-1)\ell^{d-j}-j$ non-neighbors in $U_j\setminus\{u_j^1,\ldots,u_j^\ell\}$.
We choose that vertex to be $u_j$, and we let $U'_j$ be the set of non-neighbors of $u_j$ in $U_j\setminus\{u_j^1,\ldots,u_j^\ell\}$.
Thus $\size{U'_j}\geq\ell^{d-j-1}+(d-1)\ell^{d-j}-j$.
For each $i \in [j]$, by \eqref{eq:step-neighbors} and Lemma~\ref{lem:non-neighbors}, at most $\ell-1$ vertices in $U'_j$ have fewer than $\ell^{2d-j-2}$ non-neighbors in $B(u_i)-B(u_1,\ldots,u_{i-1},u_{i+1},\ldots,u_j)$.
This gives at most $j(\ell-1)$ vertices in total (over all $i \in [j]$), which we remove from $U'_j$ to obtain $U_{j+1}$.
Thus
\begin{equation}\label{eq:step-size'}
\size{U_{j+1}}\geq\size{U'_j}-j(\ell-1)\geq \ell^{d-j-1}+(d-1)\ell^{d-j}-j\ell\text{,}
\end{equation}
and for each $u_{j+1}\in U_{j+1}$ and each $i \in [j]$, we have
\begin{equation}\label{eq:step-neighbors'}
\size{B(u_i)\setminus B(u_1,\ldots,u_{i-1},u_{i+1},\ldots,u_{j+1})}\geq\ell^{2d-j-2}\text{.}
\end{equation}
Furthermore, the inequality \eqref{eq:step-neighbors'} holds also for each $u_{j+1}\in U_{j+1}$ and $i=j+1$, because
\[\begin{split}
\size{B(u_{j+1})\setminus B(u_1,\ldots,u_j)}&=\size{B(u_{j+1})\setminus B(u_1,\ldots,u_{j-1})}-\size{B(u_j,u_{j+1})\setminus B(u_1,\ldots,u_{j-1})}\\
&\geq\size{B(u_j)\setminus B(u_1,\ldots,u_{j-1})}-\size{B(u_j,u_{j+1})\setminus B(u_1,\ldots,u_{j-1})}\\
&=\size{B(u_j)\setminus B(u_1,\ldots,u_{j-1},u_{j+1})}\\
&\geq\ell^{2d-j-2}\text{,}
\end{split}\]
where the first inequality follows from the choice of $u_j^1,\ldots,u_j^\ell$ and the second inequality is \eqref{eq:step-neighbors'} for $i=j$.
Inequalities \eqref{eq:step-size'} and \eqref{eq:step-neighbors'} show that \eqref{eq:step-size} and \eqref{eq:step-neighbors} hold for the next value of $j$.
For $j=d$, \eqref{eq:step-size} yields $\size{U_d}\geq 1$, so we can choose a vertex $u_d\in U_d$, and then \eqref{eq:step-neighbors} shows that $u_1,\ldots,u_d$ satisfy \eqref{eq:neighbors}, as claimed.

By \eqref{eq:neighbors}, we can apply Lemma~\ref{lem:independent-set} to the sets $V_i=B(u_i)\setminus B(u_1,\ldots,u_{i-1},u_{i+1},\ldots,u_d)$ for $i \in [d]$ to obtain an independent set $\{v_1,\ldots,v_d\}$ such that $v_i\in V_i$ for every $i \in [d]$.
It follows that the subgraph of $G$ induced on $r,u_1,v_1,\ldots,u_d,v_d$ is a copy of $S'_d$.
This contradiction shows that \eqref{eq:U-size} cannot hold, so $\size{U}\leq\ell^{d-1}+(d-1)\ell^d-1$.

Note that the graph $G[A\setminus U]$ does not contain $K_{k-1,\ell}$ as a subgraph, as otherwise the graph $G[A \setminus U \cup \{r\}]$, and thus $G$, contains $K_{k,\ell}$ as a subgraph.
The induction hypothesis applied to $G[A\setminus U]$ yields a vertex $v\in A\setminus U$ with degree at most $(k-2)(\ell^{d-1}+(d-1)\ell^d+\ell^{2d-2})+\ell-(k-1)$ in $G[A\setminus U]$.
The remaining neighbors of $v$ in $G$ belong to $\{r\}\cup U\cup B(v)$.
By \eqref{eq:U-def}, we have $\size{B(v)}\leq\ell^{2d-2}-1$ and thus $\size{\{r\}\cup U\cup B(v)} \leq \ell^{d-1}+(d-1)\ell^d+\ell^{2d-2}-1$.
Hence, the total degree of $v$ in $G$ is at most $(k-1) \left (\ell^{d-1}+(d-1)\ell^d+\ell^{2d-2} \right)+\ell-k$, which completes the induction step.
\end{proof}

\section{Auxiliary tools needed in the proof of Theorem~\ref{thm:main}}\label{sec:tools}

\subsection{Basic notions}

An \emph{oriented path} in an undirected graph is a path with one of its endpoints designated as the \emph{first vertex} and the other endpoint designated as the \emph{last vertex}. Thus, for an oriented path $P$ and an integer $i\in [|V(P)|]$, we can speak about the \emph{$i$th} vertex $P[i]$ on $P$, where $P[1]$ is the first vertex.

Two oriented paths $P_1$ and $P_2$ are \emph{anticomplete} if they are vertex-disjoint and no vertex on $P_1$ has any neighbor on $P_2$.
Two oriented paths $P_1$ and $P_2$ are \emph{partially anticomplete} if $P_1$ and $P_2$ are vertex-disjoint, have the same length, and the vertices $P_1[i]$ and $P_2[i]$ are non-adjacent for every $i\in [k]$, where $k=|V(P_1)|=|V(P_2)|$. A family of oriented paths is \emph{partially anticomplete} if all paths in this family are pairwise partially anticomplete. In particular, all paths in a partially anticomplete family have the same length.

For the remainder of the paper, we fix the integers $t$ and $\ell$ considered in the statement of Theorem~\ref{thm:main}.
We assume that $t$ is even, $t\geq 10$, and $\ell\geq 2$.
We also fix the constant $\epsilon\coloneqq\epsilon_{C_t}$ provided by Theorem~\ref{thm:FoxS} for the graph $C_t$.
As $C_{\geq t}$-free graphs are in particular $C_t$-free, Theorem~\ref{thm:FoxS} immediately gives the following.

\begin{corollary}\label{cor:FoxS}
Let $G$ be a $C_{\geq t}$-free graph that does not contain $K_{\ell,\ell}$ as a subgraph, and let $Y$ be a set of at least $\ell^{1/\epsilon}$ vertices of $G$.
Then there exists a subset $Y' \subseteq Y$ of size at least $|Y|^\epsilon$ that is independent in $G$.
\end{corollary}

\subsection{Clique minors}

For a vertex $v \in V(G)$ and set $X \subseteq V(G) \setminus\{v\}$, we say that $v$ and $X$ are \emph{adjacent} if $v$ has a neighbor in $X$.
A \emph{clique minor} in a graph $G$ is a family $\K$ consisting of pairwise disjoint subsets of vertices of $G$, called \emph{branch sets}, such that
\begin{itemize}[nosep]
    \item for each $K\in \K$, the graph $G[K]$ is connected; and
    \item for all distinct $K,K'\in \K$, there is a vertex in $K$, which is adjacent to $K'$.
\end{itemize}
The \emph{size} of a clique minor $\K$, denoted by $|\K|$, is the number of branch sets of $\K$.

We will use the following classic result, which asserts that excluding a clique minor implies a polynomial upper bound on degeneracy.

\begin{theorem}[Kostochka~\cite{Kostochka1982,Kostochka1984}, Thomason~\cite{thomason_1984}]\label{thm:Kostochka}
Every graph with no clique minor of size $p$ has degeneracy at most $\Oh(p \sqrt{\log p})$.
\end{theorem}

We say that a clique minor $\K$ is \emph{minimal} if all branch sets are inclusion-wise minimal with respect to $\K$ being a minor.
It is straightforward to verify that if $\K$ is minimal, then for each $K \in \K$ and each $v \in K$ at least one of the following holds:
\begin{itemize}[nosep]
    \item $v$ is a cutvertex of $G[K]$,
    \item there exists a branch set $K' \in \K \setminus \{K\}$, which is adjacent to $v$, and non-adjacent to every vertex in $K \setminus \{v\}$.
\end{itemize}
Indeed, if a vertex $v$ does not satisfy any of these conditions, then we can safely remove it from its branch set, which contradicts the minimality of $\K$.

In the next lemma we show that in $C_{\geq t}$-free graphs we can assume that each branch set has small diameter, i.e., all vertices are close to each other. Essentially the same argument was used in~\cite[Section~6, Claim~10]{deg2tw}.

\begin{lemma}\label{lem:depth}
Let $G$ be a $C_{\geq t}$-free graph and let $\K$ be a minimal clique minor in $G$ of size at least 3.
Then for each $K \in \K$, each shortest path in $G[K]$ has fewer than $t$ vertices.
\end{lemma}
\begin{proof}
For contradiction, assume that there are $K \in \K$ and two vertices $u,v \in \K$ such that a shortest $u$-$v$-path $P$ in $G[K]$ has at least $t$ vertices.
Without loss of generality we can assume that $u$ and $v$ are at maximum distance in $G[K]$,
which implies that none of them is a cutvertex in $G[K]$.
Thus by the minimality of $\K$, there are two distinct branch sets $K_u,K_v$ in $\K \setminus \{K\}$,
such that $K_u$ is adjacent to $u$ and to no other vertex of $K$,
and $K_v$ is adjacent to $v$ and to no other vertex of $K$.

Let $Q$ be a shortest $u$-$v$-path in $G$, whose all internal vertices are in $K_u \cup K_v$; it exists, as $G[K_u \cup K_v]$ is connected.
Note that $Q$ is induced and has at least two internal vertices. Furthermore, no internal vertex of $Q$ is adjacent to any internal vertex of $P$. Thus concatenating $P$ and $Q$ yields an induced cycle in $G$ with more than $t$ vertices, a contradiction.
\end{proof}

Fix a clique minor $\K$ in a graph $G$.
A vertex $x$ is \emph{full} (with respect to $\K$) if it is adjacent to all branch sets $K\in \K$ such that $x\notin K$.
In next two lemmas we show that in $C_{\geq t}$-free graphs, the existence of a large clique minor implies the existence of a large clique minor whose every branch set contains a full vertex.

\begin{lemma}\label{lem:big}
Let $G$ be a $C_{\geq t}$-free graph, $p \in N$, and let $\K$ be a minimal clique minor in $G$ of size at least $2t^3p^2$.
Then there are at least $p$ branch sets $K$ in $\K$ such that each of them contains a vertex adjacent to at least $p^2$ branch sets of~$\K$.
\end{lemma}

\begin{proof}
Let $\K' \subseteq \K$ be the subfamily of $\K$ consisting of those branch sets of $\K$ whose every vertex is adjacent to fewer than $p^2$ branch sets.
For contradiction, suppose that the statement of the lemma does not hold, i.e., $|\K'| \geq 2 t^3 p^2 -p \geq t^3 p^2$.
Recall that since $\K$ is a minimal clique minor in $G$, by Lemma~\ref{lem:depth} we have that for every $K \in \K$, every shortest path in $G[K]$ has fewer than $t$ vertices.

For each distinct $K,K' \in \K'$, we select two adjacent vertices $v_{K,K'} \in K$ and $v_{K',K} \in K'$; they exist, as $\K'$ is a clique minor.
For each triple of distinct sets $K,K',K'' \in \K'$, let $P_{K,K',K''}$ be a shortest $v_{K',K}$-$v_{K',K''}$-path fully contained in $K'$. Note that thus, $P_{K,K',K''}$ is induced and has fewer than $t$ vertices.

Let $S=(K_0,K_1,\ldots,K_{t-1})$ be a sequence of sets from $\K'$, chosen uniformly and independently at random.
Note that the probability that some set appears in $S$ more than once is at most $\binom{t}{2}/|\K'| < 0.1$.

Suppose now that the sets in $S$ are pairwise distinct.
We treat $S$ as a cyclic sequence, i.e., all arithmetic operations on indices are performed modulo $t$.
In particular, $K_{t}=K_0$.
To simplify the notation, for any distinct $a,b,c \in \{0,\ldots,t-1\}$ we write $v_{a,b}$ instead of $v_{K_a,K_b}$ and $P_{a,b,c}$ instead of $P_{K_a,K_b,K_c}$.

Define $C$ as the concatenation of the following paths in $G$:
\[
C = v_{0,t-1} - P_{t-1,0,1} - v_{0,1} - v_{1,0} - P_{0,1,2} - v_{1,2} - v_{2,1} - P_{1,2,3} - \ldots - P_{t-2,t-1,0} - v_{t-1,0}\text{.}
\]
Observe that $C$ is a cycle in $G$ that visits all sets $K \in S$ in the order prescribed by $S$.
Thus $C$ has at least $t$ vertices, but a priori it may not be induced.

Consider any $i,j \in \{0,1,\ldots,t-1\}$ that are non-consecutive modulo $t$.
Now let us estimate the probability that there is an edge in $G$ with one endpoint in $V(C) \cap K_{i}$ and the other in $V(C) \cap K_{j}$.

Observe that the choice of $P_{i-1,i,i+1}$ depends only on the choice of $K_{i-1},K_i$, and $K_{i+1}$.
In particular, it is independent of the choice of $K_j$.
As  $P_{i-1,i,i+1}$ has fewer than $t$ vertices and each of them is adjacent to fewer than $p^2$ branch sets of $\K$ (and thus also of $\K'$), we conclude that the probability that there is an edge with one endpoint in $V(P_{i-1,i,i+1})$ and second in $K_j$ is smaller than $t \cdot p^2 / |\K'|$. Summing over all choices of $i$ and $j$, we conclude that the probability that $C$ has an edge between vertices from branch sets that are non-consecutive on $S$ is less than $\binom{t}{2} \cdot t \cdot p^2/|\K'| < 0.5$.

Thus with probability at least $1-0.1-0.5 = 0.4$, the sequence $S$ contains $t$ distinct branch sets (i.e., $C$ is well defined) and the only possible chords in the cycle $C$ are between vertices of branch sets that are consecutive in $S$. So there exists at least one choice of $S$ that gives a cycle $C$ with such a property.

Now let $C'$ be a shortest cycle using only vertices of $C$ that intersects all branch sets of $S$ in the ordering given by this sequence.
Observe that since $C$ does not have chords between non-consecutive branch sets, $C'$ is induced.
Furthermore, it has at least $t$ vertices. This is a contradiction with the assumption that $G$ is $C_{\geq t}$-free.
\end{proof}

\begin{lemma}\label{lem:full}
Let $p \in \N$, $G$ be a $C_{\geq t}$-free graph, and $\K$ be a clique minor in $G$ of size at least $2 t^3p^2$.
Then there exists a clique minor $\K'$ in $G$ such that $|\K'|=p$ and every branch set of $\K'$ contains a full vertex.
Furthermore, for each $K' \in \K'$, each shortest path in $G[K']$ has fewer than $2t$ vertices.
\end{lemma}

\begin{proof}
Without loss of generality we can assume that $\K$ is minimal, as otherwise we can safely remove some vertices from branch sets of $\K$.
By Lemma~\ref{lem:big}, there are $p$ distinct branch sets $K_1,\ldots,K_p\in \K$ such that for each $i\in [p]$, there is $b_i \in K_i$ which is adjacent to at least $p^2$ branch sets of $\K$.
Let $\Nn_i$ consist of all the branch sets of $\K$ to which $b_i$ is adjacent.

For each pair $(i,j)\in [p]\times [p]$ such that $i\neq j$ we select a branch set $L_{i,j}\in \Nn_j$. Observe that we can select branch sets $L_{i,j}$ so that they are pairwise different and different from $K_1,\ldots,K_p$. Indeed, we consider the relevant pairs $(i,j)$ in any order and when choosing $L_{i,j}$ from $\Nn_j$, we avoid all the branch sets $K_1,\ldots,K_p$ and all the branch sets $L_{i',j'}$ that were previously selected. Since $|\Nn_j|\geq p^2$ and the number of branch sets to be avoided is less than $p+p(p-1)=p^2$, there is always at least one valid choice for $L_{i,j}$.

For each $i\in [p]$ we define
\[K'_i\coloneqq K_i\cup \bigcup_{j\in [p]\setminus \{i\}} L_{i,j}\text{.}\]
Let $\K'\coloneqq \{K_1',\ldots,K_p'\}$. Clearly, $\K'$ is a clique minor of size $p$ and for each $i\in [p]$, the vertex $b_i$ belongs to $K_i'$ and is full in $\K'$. 

Now let us consider some $K' \in \K'$ and two vertices $u,v \in K'$.
Let $K_u,K_v \in \K$  be, respectively, the branch sets of $\K$ containing $u$ and $v$. If $K_u=K_v$, then $u$ and $v$ can be connected within $G[K_u]$ by a shortest path with fewer than $t$ vertices; this path is in particular contained in $G[K']$.
Otherwise, recall that $G[K_u\cup K_v]$ is connected, so let $u'v'$ be an arbitrary edge with one endpoint ($u'$) belonging to $K_u$ and second ($v'$) belonging to $K_v$. As shortest paths in $G[K_u]$ and $G[K_v]$ have fewer than $t$ vertices, the concatenation of a shortest $u$-to-$u'$ path in $K_u$ and a shortest $v$-to-$v'$ path in $K_v$ gives a path in $G[K']$ from $u$ to $v$ with fewer than $2t$ vertices.  As $u$ and $v$ were chosen arbitrarily in $K'$, we conclude that every shortest path in $G[K']$ has fewer than $2t$ vertices.
\end{proof}

\subsection{VC dimension and its consequences}

A \emph{set system} is a pair $\Ss=(U,\Ff)$ where $U$ is the universe and $\Ff$ is a family of subsets of $U$. A set $Z\subseteq U$ is \emph{shattered} by $\Ss$ if for every subset $Z'$ of $Z$ there exists $F\in \Ff$ such that $F \cap Z= Z'$. The \emph{VC dimension} of $\Ss$ is the maximum size of a set shattered by $\Ss$.

We will use the well-known Sauer--Shelah lemma, stated below.

\begin{theorem}[\cite{chervonenkis1971theory,sauer1972density,shelah1972combinatorial}]\label{thm:Sauer-Shelah}
Let $k\in \N$. Every set system with universe of size $n$ and VC dimension at most $k$
contains at most $\sum_{i=0}^{k} \binom{n}{i}$ distinct sets.
\end{theorem}

The following lemma will be used several times. Here, for a set of vertices $X$ and a vertex $u\notin X$, by the neighborhood of $u$ in $X$ we mean the set of neighbors of $u$ that are contained in $X$.

\begin{lemma}\label{lem:traces0}
Let $q\in \N$ and $G$ be a $C_{\geq t}$-free graph.
Let $X$ be a set of at least $qt/2$ vertices in $G$ such that $G[X]$ is $q$-colorable.
Let $Y$ be an independent set in $G$ that is disjoint from $X$.
Then there exists a set $Y' \subseteq Y$ such that $|Y'| \geq \frac{|Y|}{|X|^{qt/2}}$ and all the vertices in $Y'$ have the same neighborhood in $X$.
\end{lemma}

\begin{proof}
Let $\Ss=(X,\Ff)$ be the set system with universe $X$, where
\[\Ff\coloneqq \{N[u]\cap X\colon u\in Y\}\text{.}\]
We verify that the VC dimension of $\Ss$ is less than $qt/2$.
Suppose for the sake of contradiction that $X$ contains a set $Z$ of size at least $qt/2$ that is shattered by $\Ss$.
Since $G[X]$ is $q$-colorable, the set $Z$ contains a subset $I$ of size at least $\frac{|Z|}{q}\geq t/2$ that is independent in $G$.
Let $z_0,\ldots,z_{t/2-1}$ be arbitrary $t/2$ vertices of~$I$. The arithmetic operations on indices of these vertices are computed modulo $t/2$.
Since $Z$ is shattered by $\Ss$, for each $i\in \{0,1,\ldots,t/2-1\}$ there exists a vertex $y_i\in Y$ that is adjacent to $z_i$ and $z_{i+1}$ and non-adjacent to all other vertices among $\{z_0,\ldots,z_{t/2-1}\}$.
Consequently, the set $\bigcup_{i \in \{0,\ldots,t/2-1\}} \{ z_i, y_i\}$ induces a $C_{t}$ in $G$, which is a contradiction.

Since the VC dimension of $\Ss$ is less than $qt/2$,
Theorem~\ref{thm:Sauer-Shelah} implies that
\[|\Ff|\leq \sum_{i=0}^{qt/2-1} \binom{|X|}{i}\leq \sum_{i=0}^{qt/2-1}|X|^i\leq |X|^{qt/2}\text{,}\]
where the last inequality follows from the fact that $|X|\geq qt/2\geq 2$.
Therefore, there must exist a set $Y'\subseteq Y$ of size at least $\frac{|Y|}{|X|^{qt/2}}$ such that all vertices $u\in Y'$ give rise to the same set $N[u]\cap X\in \Ff$, that is, the vertices in $Y'$ have the same neighborhood in $X$.
\end{proof}

We now derive two useful corollaries by combining Lemma~\ref{lem:traces0} with the assumption that the considered graph does not contain $K_{\ell,\ell}$ as a subgraph.

\begin{corollary}\label{cor:traces}
Let $G$ be a $C_{\geq t}$-free graph that does not contain $K_{\ell,\ell}$ as a subgraph.
Let $X$ be a set of at least $qt/2$ vertices of $G$ such that $G[X]$ is $q$-colorable.
Let $Y$ be an independent set in $G$ that is disjoint from $X$ and such that every vertex of $Y$ has at least $\ell$ neighbors in $X$.
Then $|Y|<\ell \cdot {|X|}^{qt/2}$.
\end{corollary}

\begin{proof}
By Lemma~\ref{lem:traces0}, there exists a set $Y' \subseteq Y$ such that $|Y'| \geq \frac{|Y|}{|X|^{qt/2}}$ and all vertices of $Y'$ have the same neighborhood $N$ in $X$.
Since $|N|\geq \ell$, we necessarily have $|Y'|<\ell$, for otherwise $G$ would contain $K_{\ell,\ell}$ as a subgraph.
Thus $|Y|\leq |Y'|\cdot {|X|}^{qt/2}<\ell \cdot  {|X|}^{qt/2}$.
\end{proof}

\begin{corollary}\label{cor:traces3}
Let $G$ be a $C_{\geq t}$-free graph that does not contain $K_{\ell,\ell}$ as a subgraph.
Let $X$ be a set of at least $qt/2$ vertices of $G$ such that $G[X]$ is $q$-colorable.
Let $Y$ be an independent set in $G$ that is disjoint from $X$ and such that $|Y|\geq \ell\cdot |X|^{qt/2}$.
Then there exist sets $X' \subseteq X$ and $Y' \subseteq Y$ such that
\[|X'|> |X|-\ell\text{,}\qquad|Y'| \geq \frac{|Y|}{|X|^{qt/2}}\text{,}\]
and no vertex of $X'$ is adjacent to any vertex of $Y'$.
\end{corollary}

\begin{proof}
By Lemma~\ref{lem:traces0}, there exists a set $Y' \subseteq Y$ such that $|Y'| \geq \frac{|Y|}{|X|^{qt/2}}\geq \ell$ and the vertices of $Y'$ have the same neighborhood $N$ in $X$.
Since $|Y'|\geq \ell$, we necessarily have $|N|<\ell$, for otherwise $G$ would contain $K_{\ell,\ell}$ as a subgraph.
Therefore, the sets $Y'$ and $X'\coloneqq X \setminus N$ satisfy the requested conditions.
\end{proof}

\section{Proof of Theorem~\ref{thm:main}}\label{sec:main}

We now proceed to the proof of our main result, Theorem~\ref{thm:main}. 
\mainthm*

For this, we fix a $C_{\geq t}$-free graph $G=(V,E)$ that does not contain $K_{\ell,\ell}$ as a subgraph.
Let $c$ be the constant hidden in $\Oh(\cdot)$-notation in Theorem~\ref{thm:Kostochka}, that is, every graph with degeneracy at least $c \cdot p \cdot \sqrt{p \log p}$ contains a clique minor of size at least $p$. Also, recall that $\epsilon=\epsilon_{C_t}$ is the constant provided by Theorem~\ref{thm:FoxS} for the graph $C_t$.
Let us introduce some constants that will be used in the proof.
\begin{align*}
R & \coloneqq \left\lceil 2t \cdot  (2t^3 \cdot \ell)^{2t^4/\epsilon^{2t}}\right\rceil\text{,} \\
N & \coloneqq (R \cdot 2t \cdot \ell)^2\text{,}\\
Z & \coloneqq 3 \left\lceil N^2 \cdot (\ell \cdot N^{t/2} )^{1/\epsilon} \right\rceil\text{,}\\
W & \coloneqq 2 \cdot t^3 \cdot Z^2\text{.}
\end{align*}
We will show that $G$ contains a vertex of degree at most
\[d\coloneqq c \cdot W^{2}\text{.}\]
Observe that thus, $d$ depends polynomially on $\ell$, that is, 
$d= \alpha(t,\epsilon) \cdot \ell^{\beta(t,\epsilon)}$ for some constants $\alpha$ and $\beta$ that depend on $t$ and $\epsilon$ (which itself also depends on $t$), and $\beta(t,\epsilon) \leq \frac{100t^5}{\epsilon^{2t+1}}$.
Since $C_{\geq t}$-free graphs with no $K_{\ell,\ell}$ as a subgraph are closed under taking induced subgraphs, this in fact shows that $G$ is $d$-degenerate, thereby proving Theorem~\ref{thm:main}.

The proof is split into the following six steps.

\paragraph{Step 1.}
For contradiction, suppose the minimum degree of $G$ is larger than $c \cdot W^2 \geq c \cdot W \cdot \sqrt{\log W}$.
By Theorem~\ref{thm:Kostochka}, we know that $G$ contains a clique minor $\K$ of size $W$.

\paragraph{Step 2.}
Since $W=2 \cdot t^3 \cdot Z^2$, by Lemma~\ref{lem:full} the graph $G$ contains a clique minor $\mathcal{K}'$ of size $Z$, such that every branch set of $\mathcal{K}'$ contains a full vertex (with respect to $\K'$), and each shortest path in the graph induced by each branch set has fewer than $2t$ vertices. 
Observe that
\begin{equation*}
 N^2 \cdot (\ell \cdot N^{t/2} )^{1/\epsilon} \geq N^{1/\epsilon} \qquad\qquad \text{and} \qquad\qquad
 N^2 \cdot (\ell \cdot N^{t/2} )^{1/\epsilon} \geq R \cdot N^2\text{,}
\end{equation*}
so $|\K'| =  Z = 3 \left \lceil  N^2 \cdot (\ell \cdot N^{t/2} )^{1/\epsilon} \right\rceil \geq  \left\lceil N^{1/\epsilon} \right\rceil + \left\lceil R \cdot N^2  + N^2 \cdot (\ell \cdot N^{t/2} )^{1/\epsilon} \right\rceil$.

\paragraph{Step 3.}
We partition the branch sets of $\K$ into two groups $\Aa$ and $\Bb$, so that 
\begin{equation*}
|\Aa| \geq \bigl\lceil N^{1/\epsilon} \bigr\rceil \qquad\qquad \text{and} \qquad\qquad
|\Bb| \geq \bigl\lceil R \cdot N^2  + N^2 \cdot (\ell \cdot N^{t/2} )^{1/\epsilon} \bigr\rceil\text{.}
\end{equation*}
For each branch set $K$ of $\Aa$, select any vertex of $K$ that is full in $\K'$ and let $A$ be the set comprising all the selected vertices.

\paragraph{Step 4. (idea)}
We argue that there is a subset $A' \subseteq A$ of size at least $t/2$ that is independent in $G$ and such that for every pair of distinct vertices $u,v\in A'$, we can find a large family $\Pp_{u,v}$ of induced paths, each with fewer than $2t$ vertices, with the following property: for each $P\in \Pp_{u,v}$, one endpoint of $P$ is adjacent to $u$, the other endpoint of $P$ is adjacent to $v$, and there are no more edges between $A'$ and $V(P)$ apart from those two. The size of each family $\Pp_{u,v}$ will be lower bounded by $R$. Moreover, all the paths in $\bigcup_{u,v\in A'} \Pp_{u,v}$ are pairwise vertex-disjoint.

\paragraph{Step 5. (idea)}
From each family $\Pp_{u,v}$, we extract a large subfamily $\Pp'_{u,v}$ that is partially anticomplete.
For this, we repeatedly apply Corollary~\ref{cor:FoxS} to the set $\Qq[i]$ for all relevant indices $i$, where $\Qq$ denotes the remaining subfamily, initially set to $\Pp_{u,v}$.
Since each path in $\Pp_{u,v}$ has fewer than $2t$ vertices, Theorem~\ref{thm:FoxS} is applied at most $2t-1$ times for one given pair $u,v$.
The size of each resulting family $\Pp'_{u,v}$ will be lower bounded by $(2t^2\cdot \ell)^{2t^4}$.

\paragraph{Step 6. (idea)}
Let us  enumerate the vertices in $A'$ as $a_0,a_1,\ldots,a_{t/2-1}$; again we perform the arithmetic operations on indices modulo $t/2$.
Our goal is to find an induced cycle with at least $t$ vertices to reach a contradiction.
For each $0 \leq i \leq t/2-1$, there is a large partially anticomplete family of paths $\Pp'_{a_i,a_{i+1}}$.
From each family $\Pp'_{a_i,a_{i+1}}$ we extract a single path $P_{a_i,a_{i+1}}$ such that the selected paths are pairwise anticomplete.
Now 
\[a_0 - P_{a_0,a_1} -a_1-P_{a_1,a_2}- a_2-P_{a_2,a_3}-a_3-\ldots -a_{t/2-1}- P_{a_{t/2-1},a_{0}} - a_0\]
is an induced cycle which contains at least $t$ vertices; a contradiction.

\bigskip

We are left with providing formal details for Steps~4,~5, and~6.

\paragraph{Step 4.} We first introduce an auxiliary result. 
An \emph{interference matrix} of order $M$ is a square $M\times M$ matrix $\Abb=[a_{ij}]_{i,j\in [M]}$ where every entry $a_{ij}$ is a subset of $[M]$ that does not contain $i$ or $j$. An interference matrix is \emph{$r$-bounded} if all its entries are sets of size at most $r$. Note that similar objects were defined in~\cite{BOUSQUET20152302}. 
We will use the following statement:
\begin{lemma}\label{lem:interference}
Let $s,r,M\in \N$ be such that $ \sqrt{M} \geq r > s^3$ and let $\Abb=[a_{ij}]$ be an $r$-bounded interference matrix of order $M$. Then there exists a subset $S\subseteq [M]$ such that $|S|=s$ and
\[a_{jj'} \cap S = \emptyset\qquad \text{for all distinct } j,j' \in S\text{.}\]
\end{lemma}
\begin{proof}
Call a triple $(i,j,k)\in [M]^3$ \emph{bad} if $k \in a_{i,j}$.
Given two distinct indices $i,j\in [M]$, if we select $k \in [M]$ uniformly at random, we have:
\[ \mathbb{P}(k \in a_{ij}) \le \frac{r}{M} \le \frac{r}{r^2} =\frac{1}{r}\text{.} \]
Thus, if we select a set $S$ of $s$ distinct indices uniformly at random, then we have:
\[ \mathbb{E}(\text{number of bad triples in }S) \le s^3 \cdot \frac{1}{r} < 1\text{.} \]
We conclude that there must exist a set $S$ of size $s$ for which there is no bad triple, as desired.
\end{proof}

We now proceed to the implementation of Step~4.
Recall that from Step~3 we have obtained a set $A$ of vertices and a family $\Bb$ of pairwise disjoint, connected sets of vertices of $G$ such that $A$ is disjoint with $\bigcup \Bb$ and each vertex in $A$ has a neighbor in each set in $\Bb$.
Furthermore, we have the following lower bounds on the cardinalities of $A$ and $\Bb$. 
\begin{align}
|A| & \geq \bigl\lceil N^{1/\epsilon} \bigr\rceil\text{,} \label{eq:sizeA} \\
|\Bb| & \geq \bigl\lceil R \cdot N^2  + N^2 \cdot (\ell \cdot N^{t/2} )^{1/\epsilon} \bigr\rceil\text{.} \label{eq:sizeB}
\end{align}
The following claim encapsulates the outcome of this step.

\begin{claim}\label{cl:step4}
There exists a subset $A' \subseteq A$ of size at least $t/2$
such that $A'$ is independent in $G$ and for every pair of distinct vertices $u,v\in A'$, there exists a family $\Pp_{u,v}$ of $R$ induced paths, each with fewer than $2t$ vertices, satisfying the following property: for each $P\in \Pp_{u,v}$, one endpoint of $P$ is adjacent to $u$, the other endpoint of $P$ is adjacent to~$v$, and there are no edges between $A'$ and $V(P)$ apart from those two. Moreover, all the paths in $\bigcup_{u,v\in A'} \Pp_{u,v}$ are pairwise vertex-disjoint.
\end{claim}
\begin{proof}
Since $|A|\geq N^{1/\epsilon}\geq \ell^{1/\epsilon}$, by Corollary~\ref{thm:FoxS} combined with \eqref{eq:sizeA} we may infer that there is an independent set $A'' \subseteq A$ satisfying
\[|A''|\geq |A|^\epsilon \geq (R \cdot t \cdot \ell)^2=N\text{.}\]
By removing some vertices if necessary, we may assume $|A''|=N$.

We partition the family $\Bb$ into $\binom{N}{2}$ groups, each of size at least  $\left\lfloor \frac{|\Bb|}{\binom{N}{2}} \right\rfloor \geq \frac{|\Bb|}{N^2}$. Now, to each pair of distinct vertices $u,v  \in A''$ we assign a distinct group and denote it by $\Bb_{u,v}$.

Consider a pair $u,v\in A''$ and a branch set $C \in \Bb_{u,v}$. Since $u$ and $v$ are full vertices in the clique minor~$\K'$,
there exists a $u$-$v$ path whose all internal vertices are contained in $C$.
Let $R_C$ be a shortest such path, note that it is induced and has fewer than $2t$ internal vertices, as the subpath of $R_C$ contained in $C$ is a shortest path in $G[C]$.

Let $\Rr_{u,v}$ be the set containing, for each $C \in \Bb_{u,v}$, one such path $R_C$ with $u$ and $v$ removed. 
Clearly, $|\Rr_{u,v}| =  |\Bb_{u,v}| \geq  \frac{|\Bb|}{N^2}$ and  all the paths in $\bigcup_{u,v\in A''} \Rr_{u,v}$ are pairwise vertex-disjoint.

Let $Y$ be the set of those vertices in $\bigcup_{P\in \Rr_{u,v}} V(P)$ that have at least $\ell$ neighbors in $A''$.
We claim that $|Y|<(\ell \cdot N^{t/2})^{1/\epsilon}$.
Indeed, otherwise, by Corollary~\ref{cor:FoxS}, there would be an independent set $Y' \subseteq Y$ such that $|Y'|\geq |Y|^\epsilon \geq \ell \cdot N^{t/2}$.
However, by Corollary~\ref{cor:traces}, the size of such a set $Y'$ must be smaller than $\ell \cdot |A''|^{t/2} =\ell \cdot N^{t/2}$, a contradiction.

Note that the number of paths in $\Rr_{u,v}$ which do not contain any vertex from $Y$ is at least $|\Rr_{u,v}| - |Y|$. On the other hand, by~\eqref{eq:sizeB} we have
\[
|\Rr_{u,v}| - |Y| \geq \frac{|\Bb|}{N^2} - (\ell \cdot N^{t/2})^{1/\epsilon} \geq  \frac{ R \cdot N^2 +  N^2 \cdot (\ell \cdot N^{t/2} )^{1/\epsilon}}{N^2} -  (\ell \cdot N^{t/2})^{1/\epsilon} = R\text{.}
\] 
Therefore, we may select $\Pp_{u,v}$ to be any subfamily of $\Rr_{u,v}$ consisting of $R$ paths disjoint with $Y$. We also denote $S_{u,v}\coloneqq \bigcup_{P\in \Pp_{u,v}} V(P)$.

Consider a path $P \in \Pp_{u,v}$ and recall that it has fewer than $2t$ vertices.
Since every vertex of $P$ is adjacent to fewer than $\ell$ vertices in $A''$, we conclude that the number of vertices of $A''$ with a neighbor in $V(P)$ is at most $(2t-1)(\ell-1)$. Thus, the number of vertices of $A''$ with a neighbor in $S_{u,v}$ is at most $R \cdot (2t-1)(\ell-1) < R \cdot 2t \cdot \ell$.

Recalling that $|A''|=N$, let us enumerate vertices of $A''$ as $v_1,v_2,\ldots,v_{N}$.
Let $\Abb=[a_{i,j}]_{i,j\in [N]}$ be the $N \times N$ matrix where for $i \neq j$, the entry $a_{i,j}$ is the set of those $i' \in [N] \setminus \{i,j\}$ such that $v_{i'}$ has a neighbor in $S_{v_i,v_j}$. The diagonal entries of $\Abb$ are empty sets.
We observe that $\Abb$ is an $r$-bounded interference matrix of order $N$, where 
\[
r \coloneqq R \cdot 2t \cdot \ell = \sqrt{ (R \cdot 2t \cdot \ell)^2 } =  \sqrt{N}\text{.}
\]
Thus, we may apply Lemma~\ref{lem:interference} to $\Abb$ for $a \coloneqq \lfloor (r-1)^{1/3} \rfloor$. In this way, we obtain a subset $A' \subseteq A''$ such that for every pair of distinct vertices $u,v \in A'$, the vertices in $S_{u,v}$ have no neighbors in $A' \setminus \{u,v\}$, and 
\[
|A'|=a= \lfloor ( R \cdot 2t \cdot \ell - 1)^{1/3} \rfloor \geq t/2\text{.}
\]
We conclude that the set $A'$, along with families $\{\Pp_{u,v}\}_{u,v \in A'}$, satisfies all the required properties.
\end{proof}

\paragraph{Step 5.}
At this point, we have constructed an independent set $A'$ of size at least $t/2$ and, for each pair of distinct vertices $u,v \in A'$, a suitable family of induced paths $\Pp_{u,v}$ of size $R$.
Our current goal is to extract a sufficiently large subfamily $\Pp'_{u,v}\subseteq \Pp_{u,v}$ that is partially anticomplete. 
The main idea is captured by the following lemma.

\begin{lemma}\label{lem:makepartiallyanticompleye}
Let $G$ be a $C_{\geq t}$-free graph that does not contain $K_{\ell,\ell}$ as a subgraph and
let $\Pp$ be a family of at least $\ell^{1/\epsilon^{k}}$ vertex-disjoint, ordered induced paths in $G$, each with exactly $k$ vertices.
Then there exists a subfamily $\Pp' \subseteq \Pp$,
such that $|\Pp'| \geq |\Pp|^{\epsilon^{k}}$ and $\Pp'$ is partially anticomplete.
\end{lemma}
\begin{proof}
We prove the statement by induction on $k$.
Suppose first that $k=1$, so each path in $\Pp$ is a single vertex.
Let $W=\bigcup_{P\in \Pp} V(P)$.
By the assumption we know that $|W| = |\Pp| \geq \ell^{1/\epsilon}$, so by Corollary~\ref{cor:FoxS} there exists an independent set $W' \subseteq W$ of size at least $|\Pp|^\epsilon$. Thus $W'$, or more formally the family of one-vertex paths, each consisting of a distinct element of $W'$, satisfies the statement.

So suppose that $k \geq 2$ and the lemma holds for all $k'<k$.
For a path $P \in \Pp$, let $P^*$ be the path obtained from $P$ by removing the last vertex.
Let $\Qq \coloneqq \{P^* \colon P \in \Pp\}$. Clearly we have $|\Qq| = |\Pp| \geq \ell^{1/\epsilon^k}$, hence by induction there exists a partially anticomplete subfamily $\Qq' \subseteq \Qq$ of size at least $|\Pp|^{\epsilon^{k-1}}$.

Let $Y$ be the set consisting of all the last vertices of those paths $P \in \Pp$ for which $P^* \in \Qq'$. Since $|Y| = |\Qq'| \geq |\Pp|^{\epsilon^{k-1}} \geq \ell^{1/\epsilon}$, by Corollary~\ref{cor:FoxS} there exists an independent set $Y' \subseteq Y$ of size at least $|Y|^\epsilon \geq |\Pp|^{\epsilon^{k}} $. It is straightforward to verify that the family $\Pp' \subseteq \Pp$ consisting of all paths from $\Pp$ whose last vertex belongs to $Y'$, satisfies all the required properties.
\end{proof}

Lemma~\ref{lem:makepartiallyanticompleye} yields the following statement, which summarizes the outcome of Step~5.

\begin{claim}
For every distinct $u,v \in A'$, there is a family $\Pp_{u,v}' \subseteq \Pp_{u,v}$ of size at least  $(2t^2\cdot \ell)^{2t^4}$, which is partially anticomplete.
\end{claim}
\begin{proof}
Fix $u,v \in A'$ and recall that $|\Pp_{u,v}|\geq R = \left\lceil 2t \cdot  (2t^3 \cdot \ell)^{2t^4/\epsilon^{2t}}\right\rceil$.
Furthermore, each path in $\Pp_{u,v}$ has fewer than $2t$ vertices.
Therefore, there exists $k \in [2t-1]$ and a subfamily $\Pp_{u,v}'' \subseteq \Pp_{u,v}$ consisting of at least $R/(2t-1)\geq (2t^3 \cdot \ell)^{2t^4 / \epsilon^{2t}}$ paths from $\Pp_{u,v}$ such that each path in $\Pp_{u,v}''$ has exactly $k$ vertices.
Observe that
\[|\Pp_{u,v}''| \geq (2t^3 \cdot \ell)^{2t^4 / \epsilon^{2t}} \geq \ell^{1/\epsilon^{2t}} \geq \ell^{1/\epsilon^{k}}\text{.}\]
Therefore, we may apply Lemma~\ref{lem:makepartiallyanticompleye} to the family $\Pp_{u,v}''$, where each path in $\Pp''_{u,v}$ is ordered so that the first vertex is the endpoint adjacent to $u$ and the last vertex is the endpoint adjacent to $v$.
This yields a partially anticomplete subfamily $\Pp_{u,v}' \subseteq \Pp_{u,v}''$ of size at least
\[
|\Pp_{u,v}''|^{\epsilon^{k}} \geq (2t^3 \cdot \ell)^{2t^4 \cdot \epsilon^{k} / \epsilon^{2t}} \geq (2t^3 \cdot \ell)^{2t^4}\text{.}
\]
This completes the proof.
\end{proof}

\paragraph{Step 6.} Finally, from each family $\Pp'_{u,v}$ we wish to extract one path so that all the selected paths are pairwise anticomplete.
Let us start with the following auxiliary lemma.

\begin{lemma}\label{lem:newtraces2}
Let $G$ be a $C_{\geq t}$-free graph that does not contain $K_{\ell,\ell}$ as a subgraph.
Let $\Pp$ and $\Qq$ be two partially anticomplete families of ordered induced paths in $G$,
each with fewer than $2t$ vertices,
such that any two paths in $\Pp \cup \Qq$ are vertex-disjoint.
Suppose that $|\Qq|\geq L \cdot |\Pp|^{(2t-1)^2t/2}$ for some $L\geq \ell$.
Then there are subfamilies $\Pp' \subseteq  \Pp$ and $\Qq' \subseteq  \Qq$
such that 
\[|\Pp'| \geq |\Pp| - (\ell-1)(2t-1)\text{,}\qquad\qquad |\Qq'| \geq  L\text{,}\]
and every path in $\Pp'$ is anticomplete to every path in $\Qq'$.
\end{lemma}
\begin{proof}
Let $k_\Pp\in [2t-1]$ be the common number of vertices on each path from $\Pp$, and similarly define $k_\Qq$ for $\Qq$.
Recall that $\Pp$ is partially anticomplete, hence for each $i\in [k_\Pp]$ the set $\Pp[i]$ is independent in $G$.  
In particular, this implies that the graph $G[\bigcup_{P\in \Pp} V(P)]$ is $(2t-1)$-colorable.
Analogous conclusions can be drawn about the family $\Qq$.

Set $X_0 \coloneqq\bigcup_{P\in \Pp} V(P)$ and $\Qq_0 \coloneqq \Qq$.
We extract $\Pp'$ and $\Qq'$ iteratively, in $k_\Qq$ rounds.
In round $i$ we will obtain a subset $X_i \subseteq X_{i-1}$ and a subfamily $\Qq_i \subseteq \Qq_{i-1}$
satisfying the following  properties:
\begin{enumerate}[label={(P\arabic*)}]
\item\label{p:Xi} $|X_i| \geq |X_0| - i \cdot (\ell-1)$,
\item\label{p:Qi} $|\Qq_i| \geq L \cdot \left(|X_0|^{(2t-1)t/2}\right)^{(2t-1-i)}$, and
\item\label{p:anti} there are no edges with one endpoint in $X_i$ and the other in $\bigcup_{j=1}^{i} \Qq_{i}[j]$.
\end{enumerate}
Suppose for a moment that we construct $X_{k_{\Qq}}$ and $\Qq_{k_{\Qq}}$ that satisfy the properties listed above.
Let $\Pp'$ to be the set of those paths in $\Pp$, whose vertex sets are contained in $X_{k_{\Qq}}$.
Since $|X_0 \setminus X_{k_{\Qq}}| \leq k_\Qq(\ell-1)\leq  (2t-1)(\ell-1)$ and paths in $\Pp$ are vertex-disjoint, we have $|\Pp'| \geq |\Pp|-(2t-1)(\ell-1)$.
Similarly, if we define $\Qq'\coloneqq \Qq_{k_\Qq}$, then $|\Qq'|\geq L\cdot \left(|X_0|^{(2t-1)t/2}\right)^{(2t-1-k_\Qq)}\geq L$. 
Finally, the paths in $\Pp'$ are anticomplete to the paths in $\Qq'$, so $\Pp'$ and $\Qq'$ satisfy all the required conditions.

Observe that $X_0$ and $\Qq_0$ satisfy properties~\ref{p:Xi},~\ref{p:Qi}, and~\ref{p:anti} for $i=0$.
So let $i \in [k_{\Qq}]$ and suppose that we have already constructed $X_{i-1}$ and $\Qq_{i-1}$.
Let $Y\coloneqq \Qq_{i-1}[i]$. By~\ref{p:Qi} for round $i-1$, we have
\[|Y| \geq L \cdot \left(|X_0|^{(2t-1)t/2}\right)^{(2t-1-(i-1))} \geq L \cdot |X_0|^{(2t-1)t/2} \geq \ell \cdot |X_0|^{(2t-1)t/2} \geq \ell \cdot |X_{i-1}|^{(2t-1)t/2}\text{.}\]
We conclude that the sets $X_{i-1}$ and $Y$ satisfy the prerequisites of Corollary~\ref{cor:traces3} for $q=2t-1$. Therefore, there exist $X' \subseteq X_{i-1}$ and $Y' \subseteq Y$ such that no vertex of $X'$ is adjacent to any vertex of $Y'$, 
\[ |X'| \geq |X_{i-1}|-(\ell-1) \geq \bigl( |X_0| - (i-1) \cdot (\ell-1) \bigr) - (\ell-1) = |X_0| - i \cdot (\ell-1) \]
and
\begin{align*} |Y'| & \geq \frac{|Y|}{|X_{i-1}|^{(2t-1)t/2}} \geq \frac{L \cdot \left(|X_0|^{(2t-1)t/2}\right)^{(2t-1-(i-1))}}{|X_{i-1}|^{(2t-1)t/2}} \\
& \geq \frac{L \cdot \left(|X_0|^{(2t-1)t/2}\right)^{(2t-1-(i-1))}}{|X_{0}|^{(2t-1)t/2}}= L \cdot \left(|X_0|^{(2t-1)t/2}\right)^{(2t-1-i)}\text{.}\end{align*}
We may now define $\Qq_i$ as the set of those paths in $\Qq_{i-1}$ whose $i$th vertex belongs to $Y'$. Clearly sets $X_i$ and $\Qq_i$ satisfy properties ~\ref{p:Xi},~\ref{p:Qi}, and~\ref{p:anti}. Thus the proof is complete.
\end{proof}

From Lemma~\ref{lem:newtraces2} we may derive the following statement.

\begin{lemma}\label{lem:Pt_manypaths}
Let $G$ be a $C_{\geq t}$-free graph that does not contain $K_{\ell,\ell}$ as a subgraph.
Suppose $k \leq t$ is a positive integer and $\Qq_1,\ldots,\Qq_k$ are partially anticomplete families of ordered induced paths in $G$,
each with fewer than $2t$ vertices,
such that all the paths in $\bigcup_{i=1}^k \Qq_i$ are pairwise vertex-disjoint and $|\Qq_i|\geq (2t^2\cdot \ell)^{k\cdot 2t^3}$ for each $i\in [k]$.
Then for each $i \in [k]$ one can select a path $Q_i \in \Qq_i$ so that the paths $\{Q_i\colon i\in [k]\}$ are pairwise anticomplete.
\end{lemma}
\begin{proof}
We proceed by induction on $k$. For $k=1$ we may choose an arbitrary path $Q_1 \in \Qq_1$.
So suppose $k \geq 2$ and the claim holds for all $k'<k$.

Select a subfamily $\Qq_1' \subseteq \Qq_1$ of size $2t^2\cdot \ell$.
We will iteratively apply Lemma~\ref{lem:newtraces2} to $\Qq_1'$ and $\Qq_i$, for $i$ from $2$ to $k$, deleting some paths of $\Qq_1'$ at every round.
More formally, at round $i \in \{2,\ldots,k\}$, let us consider the subfamily $\Qq'_{1,i-1}\subseteq \Qq'_1$ consisting of paths not removed in the previous rounds.
Initially $\Qq'_{1,1}=\Qq'_{1}$.

We proceed to the description of round $i$. Note that 
\begin{align*}
    |\Qq_{i}| & \geq (2t^2 \cdot \ell)^{k \cdot 2t^3} = (2t^2 \cdot \ell)^{(k-1) \cdot 2t^3} \cdot (2t^2 \cdot \ell)^{2t^3} \\
    & \geq (2t^2 \cdot \ell)^{(k-1) \cdot 2t^3} \cdot |\Qq'_{1,i-1}|^{2t^3} \geq (2t^2 \cdot \ell)^{(k-1) \cdot 2t^3} \cdot  |\Qq'_{1,i-1}|^{(2t-1)^2t/2}\text{.}
\end{align*}
Hence, we may apply Lemma~\ref{lem:newtraces2} to the families $\Qq'_{1,i-1}$ and $\Qq_{i}$. Thus we
extract subfamilies $\Qq'_{1,i}\subseteq \Qq'_{1,i-1}$ and $\Qq'_{i}\subseteq \Qq_i$
such that every path from $\Qq'_{1,i}$ is anticomplete to every path from $\Qq'_i$,
and we have $|\Qq'_{1,i}| \geq |\Qq_{1,i-1}| - 2t \cdot \ell$ and $|\Qq'_i| \geq  (2t^2 \cdot \ell)^{(k-1) \cdot 2t^3}$.
Thus, after the last step, we obtain a subfamily $\Qq'_{1,k} \subseteq \Qq_1$ of size at least $|\Qq_1'|- (k-1) \cdot 2t \cdot \ell = 2t^2 \cdot \ell - (k-1) \cdot 2t \cdot \ell > 0$.
Let $Q_1$ be an arbitrary path from $\Qq'_{1,k}$.

Now, since for each $i \in \{2,\ldots,k\}$ the family $\Qq'_i$ has size at least $(2t^2 \cdot \ell)^{(k-1) \cdot 2t^3}$, we can apply the induction hypothesis for $k-1$ and the families $\Qq_2',\ldots,\Qq_k'$. Thus we obtain pairwise anticomplete paths $Q_2, \ldots, Q_{k}$, chosen from $\Qq'_2,\ldots, \Qq'_{k}$, respectively.
Since $Q_1$ is anticomplete to all the paths in $\bigcup_{i=2}^k \Qq'_i$, the path $Q_1$ is in particular anticomplete to each other $Q_i$. We obtain a collection as desired.
\end{proof}

We are now in a position to complete Step~6. Recall that in Steps~4 and~5 we have constructed an independent set $A'$ of size at least $t/2$ and, for each pair of distinct vertices $u,v \in A'$, a partially anticomplete family of induced paths $\Pp'_{u,v}$ of size at least $(2t^2\cdot \ell)^{2t^4}$. Moreover, as guaranteed by Claim~\ref{cl:step4}, all the paths in $\bigcup_{u,v\in A'}\Pp'_{u,v}$ are pairwise vertex-disjoint. Select arbitrary distinct vertices $a_0,\ldots,a_{t/2}\in A'$. We may now apply Lemma~\ref{lem:Pt_manypaths} to the families $\Pp'_{a_0,a_1},\Pp'_{a_1,a_2},\ldots,\Pp'_{a_{t/2-2},a_{t/2-1}},\Pp'_{a_{t/2-1},a_0}$, noting that each of these families is sufficiently large for the prerequisites of Lemma~\ref{lem:Pt_manypaths} to be satisfied. Thus we may construct a pairwise anticomplete collection of paths $P_{a_0,a_1},P_{a_1,a_2},\ldots,P_{a_{t/2-2},a_{t/2-1}},P_{a_{t/2-1},a_0}$, selected from $\Pp'_{a_0,a_1},\Pp'_{a_1,a_2},\ldots,\Pp'_{a_{t/2-2},a_{t/2-1}},\Pp'_{a_{t/2-1},a_0}$, respectively.
It now remains to observe that
\[
a_0-P_{a_0,a_1}- a_1-P_{a_1,a_2}-a_2-\ldots -a_{t/2-2}- P_{a_{t/2-2},a_{t/2-1}}-a_{t/2-1}- P_{a_{t/2-1,a_0}} - a_{0}
\]
is an induced cycle on at least $t$ vertices in $G$. This contradiction completes the proof of Theorem~\ref{thm:main}.

\section{Treedepth of \texorpdfstring{$P_t$}{P\_t}-free graphs without \texorpdfstring{$K_{\ell,\ell}$}{K\_\{l,l\}} as a subgraph}\label{sec:td}

Recall that the treewidth $\tw(G)$ of a graph $G$ is the minimum $k$ such that $G$ admits a {\em{tree decomposition}} of width at most $k$: a tree $T$ with each node $x$ associated with a {\em{bag}} $\beta(x)\subseteq V(G)$ of size at most $k+1$ such that for every edge $uv\in E(G)$ there exists $x\in V(T)$ satisfying $u,v\in \beta(x)$, and for every $u\in V(G)$ the set $\{x\in V(T)~|~u\in \beta(x)\}$ is non-empty and connected in $T$. Similarly, the treedepth $\td(G)$ of $G$ is the minimum height of an {\em{elimination forest}} of $G$: a rooted forest $F$ on the same vertex set of $G$ where for every edge $uv\in E(G)$, the vertices $u$ and $v$ are bound by the ancestor/descendant relation in $F$. 

It is well known that for every $n$-vertex graph $G$, we have
\[\tw(G)+1\leq \td(G)\leq (\tw(G)+1)\cdot \log n\text{;}\]
see e.g.~\cite[Section~6.4]{sparsity}. In general, the right inequality needs to involve a factor dependant on the vertex count $n$, because the $n$-vertex path has treewidth $1$ and treedepth as high as $\lceil \log (n+1)\rceil$. We now show that in the setting of $P_t$-free graphs, we can give an alternative bound that does not involve~$n$.

\begin{lemma}\label{lem:Pt-td-tw}
If $G$ is a $P_t$-free graph, then
\[\td(G)\leq (\tw(G)+1)^{t-1}\text{.}\]
\end{lemma}

Note that Theorem~\ref{thm:td} follows directly from combining Corollary~\ref{coro:tw} with Lemma~\ref{lem:Pt-td-tw}.

\medskip

For the proof of Lemma~\ref{lem:Pt-td-tw}, we will use the characterization of treewidth and treedepth via the {\em{generalized coloring numbers}}, introduced by Kierstead and Young~\cite{KiersteadY03}, which we define now.

Let $G$ be a graph and let $\preceq$ be a linear order on the vertex set of $G$. For two vertices $u\preceq v$ and a positive integer $r$, we shall say that $u$ is {\em{strongly $r$-reachable}} from $v$ if there exists a path $P$ of length at most $r$ with endpoints $u$ and $v$ such that for every internal vertex $w$ of $P$ we have $v\preceq w$. The notion of {\em{weak $r$-reachability}} is defined in the same way, except that we only require that every internal vertex $w$ of $P$ satisfies $u\preceq w$; that is, we allow $P$ to use vertices placed between $u$ and $v$ in the ordering $\preceq$. Note that we consider every vertex to be $r$-reachable from itself, both strongly and weakly.

Now, the {\em{strong $r$-coloring number}} of $G$, denoted $\scol_r(G)$, is the minimum integer $k$ such that there is a linear order of the vertices of $G$ in which every vertex strongly $r$-reaches at most $k$ vertices in total (including itself). The weak $r$-coloring number is defined in the same way, but using weak $r$-reachability instead of strong. While clearly the weak $r$-coloring number upper bounds the strong $r$-coloring number, it is not hard to prove that these two parameters are functionally equivalent.

\begin{lemma}[{\cite[Lemma~3]{KiersteadY03}}]\label{lem:col-wcol}
For every graph $G$ and integer $r$, we have
\[\wcol_r(G)\leq \scol_r(G)^r\text{.}\]
\end{lemma}

We may also define the limit variants of the strong and the weak coloring number, $\scol_{\infty}(G)$ and $\wcol_{\infty}(G)$, where we do not impose any restriction on the length of paths $P$ witnessing reachability. It appears that these two notions essentially coincide with treewidth and treedepth, respectively.

\begin{lemma}[see e.g. Chapter 1, Theorems 1.17 and 1.19 of~\cite{notes}]\label{lem:col-wcol-tw-td}
For every graph $G$, we have
\[\tw(G)=\scol_{\infty}(G)-1\qquad\qquad\text{and}\qquad\qquad \td(G)=\wcol_{\infty}(G)\text{.}\]
\end{lemma}

With Lemmas~\ref{lem:col-wcol} and~\ref{lem:col-wcol-tw-td} recalled, we can give a proof of Lemma~\ref{lem:Pt-td-tw}.

\begin{proof}[Proof of Lemma~\ref{lem:Pt-td-tw}]
Observe that a path witnessing strong or weak reachability can be always assumed to be induced, because otherwise we can shortcut it. Hence, from the $P_t$-freeness of $G$ we infer that
\[\scol_{\infty}(G)=\scol_{t-1}(G)\qquad\qquad\textrm{and}\qquad\qquad \wcol_{\infty}(G)=\wcol_{t-1}(G)\text{.}\]
Therefore, by Lemmas~\ref{lem:col-wcol} and~\ref{lem:col-wcol-tw-td} we have
\[\td(G)=\wcol_{\infty}(G)=\wcol_{t-1}(G)\leq \scol_{t-1}(G)^{t-1}= \scol_{\infty}(G)^{t-1}= (\tw(G)+1)^{t-1}\text{.}\qedhere\]
\end{proof}

\section{Conclusion}
During the reviewing process of our paper, Scott, Seymour, and Spirkl~\cite{ScottSS21-imp} reproved and extended our Theorem~\ref{thm:main} (using a different approach). In particular, they showed the following result.

\begin{theorem}[Scott, Seymour, and Spirkl~\cite{ScottSS21-imp}]
For every tree $H$, every $H$-free graph $G$ that does not contain $K_{\ell,\ell}$ as a subgraph has degeneracy at most $f(H,\ell)$,
where $f$ depends polynomially on $\ell$.
\end{theorem}

Recall that a similar result cannot hold if $H$ is not a tree, as there are graphs with arbitrarily large girth and chromatic number~\cite{erdos1959graph}.
However, it might still be true if we exclude $H$ and all its subdivisions. Note that $P_t$-free graphs are exactly $P_t$-subdivision-free graphs and $C_{\geq t}$-free graphs are exactly $C_t$-subdivition-free graphs. We suggest the following conjecture.

\begin{conjecture}
For every graph $H$, every $H$-subdivision-free graph $G$ that does not contain $K_{\ell,\ell}$ as a subgraph has degeneracy at most $f(H,\ell)$,
where $f$ depends polynomially on $\ell$.
\end{conjecture}

Recall that by the already mentioned result by K\"uhn and Osthus~\cite{DBLP:journals/combinatorica/KuhnO04a}, the degeneracy of $H$-subdivision-free graphs  that do not contain $K_{\ell,\ell}$ as a subgraph is bounded in terms of $H$ and $\ell$, but the known bound is superpolynomial in $\ell$.

Let us also point out that in another rich family of graph classes where $\chi$-boundedness is studied, i.e., geometric intersection graphs, changing the parameter to the size of a largest balanced biclique makes the problem easy. Indeed, by Theorem~\ref{thm:FoxPach}, every string graph that does not contain $K_{\ell,\ell}$ as a subgraph has chromatic number bounded by $\Oh(\ell \log \ell)$.

\paragraph*{Acknowledgements.} A significant part of the work leading to the results presented in this paper was carried out during the Structural Graph Theory Workshop, held in Gułtowy in June 2019, as well as Dagstuhl Seminar 19271 {\em{Graph Colouring: from Structure to Algorithms}}.
The former workshop was supported by a project that has received funding from the European Research Council (ERC) under the European Union's Horizon 2020 research and innovation programme under grant agreement No 714704 (PI: Marcin Pilipczuk).
We acknowledge the extraordinarily inspiring and productive atmosphere at both these events. We also thank Archontia Giannopoulou, Carla Groenland, Tereza Klimo\v{s}ov\'a, Piotr Micek, and Irene Muzi for helpful discussions at the preliminary stages of this project.

\bibliographystyle{abbrv}
\bibliography{main}

\end{document}